\documentclass[graybox]{svmult}

\usepackage{mathptmx}       
\usepackage{helvet}         
\usepackage{courier}        
\usepackage{type1cm}        
\usepackage{amssymb}
\usepackage{mathrsfs}
\usepackage{amsmath}

\usepackage{makeidx}         
\usepackage{graphicx}        
\graphicspath{{Figures/}}
\usepackage{multicol}        
\usepackage[bottom]{footmisc}

\makeindex             

\begin{document}

\title*{  The weak type $(1,p)$  for  convolution operators on locally compact  groups}

\author{Duv\'an Cardona}
\institute{Duv\'an Cardona \at Pontificia Universidad Javeriana, Mathematics Department, Bogot\'a-Colombia. \email{duvanc306@gmail.com, cardonaduvan@javeriana.edu.co}}
%
%
\maketitle

\abstract*{In this paper we provide necessary and sufficient conditions for the $\textnormal{weak}(1,p)$ boundedness, $1< p<\infty,$  of convolution operators on locally compact (Hausdorff) topological groups. So, we generalize a classical result due to Sobolev-Hardy-Littlewood and Stepanov. Applications to Fourier multipliers on  Lie groups also are given. MSC2010: 42B15 (primary), 22E30 (secondary).}

\abstract{In this paper we provide necessary and sufficient conditions for the $\textnormal{weak}(1,p)$ boundedness, $1< p<\infty,$  of convolution operators on locally compact (Hausdorff) topological groups. So,  we generalize a classical result due to Sobolev-Hardy-Littlewood and Stepanov.  Applications to Fourier multipliers on  Lie groups also are given. MSC2010: 42B15 (primary), 22E30 (secondary). }

\section{Introduction}

Let us consider a locally compact (Hausdorff) topological group $G,$ endowed with its unique (up by constants) left invariant Haar measure $d\mu,$ which we also  assume inversion invariant.  Given a convolution operator of the form
\begin{equation}\label{T}
  Tf(x)\equiv  k*f(x):=\int\limits_{G}k(y)f(y^{-1}x)d\mu(y),\,\,f\in C_{0}(G)
\end{equation}
 defined on the set of  functions with compact support, in this paper we provide necessary and sufficient conditions on the kernel $k,$ in order that $T$ can be extended to a bounded operator from $L^{1}(G)$ into $L^{p,\infty}(G),$ $1<p<\infty.$ This means that the following inequality
 \begin{equation}\label{WT}
 \Vert Tf \Vert_{L^{p,\infty}(G)}:=\sup_{\alpha>0}  \mu\{s\in G:|Tf(s)|>\alpha\}^{\frac{1}{p}}\leq C\Vert f \Vert_{L^1(G)}
 \end{equation}
holds true for some fixed constant $C$ and every $f\in L^1(G).$ This problem has been solved for $G=\mathbb{R}^n$ by Sobolev, Hardy, Littlewood and  Stepanov (see, e.g. \cite{Stepanov}). They showed that on $\mathbb{R}^n,$ a  convolution operator $T$ is of weak type $(1,p),$ $1<p<\infty,$ if and only if its convolution kernel  belongs to $L^{p,\infty}(\mathbb{R}^n).$ Our main goal here is to generalize the Sobolev-Hardy-Littlewood-Stepanov result to every compact locally (Hausdorff) topological group (see Theorem \ref{Thm}). The characterization of the  weak $(1,1)$ boundedness for convolution operators, is much more complicated to obtain,  and in general one only have necessary conditions depending on the type of groups as well as the structure of the convolution operator under consideration  (see, e.g., the books by Stein, Grafakos and Duoandikoetxea \cite{Duo,GrafakosBook}and \cite{Ste} and its historical notes).

The problem of find satisfactory conditions in order to assure the $L^p$-boundedness of convolution operators is classical. Convolution operators with singular kernels were treated trough of the  Calder\'on and Zygmund  work (see e.g., \cite{CalderonZygmund}).  Necessary and sufficient conditions for the boundedness of convolution operators on $L^p$ spaces are known for $p=1$ and $p=2$ (see, e.g., Grafakos \cite{GrafakosBook}). For convolution operators, there exist nice connections between the weak $(q,p)$ boundedness (i.e. boundedness from $L^q$ into $L^{p,\infty}$) and its strong $(q,p)$ boundedness (i.e. boundedness from $L^q$ into $L^{p}$). In 1982, G. Pisier conjectured  that every convolution operator bounded on $L^{p,\infty}(G)$ is automatically bounded on $L^p(G)$ if $1<p<2$ and $G$ is compact. This conjecture was proved by H. Shteinberg \cite{Shteinberg}. 

It was proved by L. Colzani in \cite{Colzani} that convolution operators  $T$ bounded on $L^{p,\infty}$ also are bounded on $L^p,$ when the group is non-compact; Colzani also proves that  a bounded convolution operator on the Lorentz space $L^{p,q}(G),$ $0<p<1$ and $0<q\leq \infty$ is a convolution operator with  a discrete measure $m=\sum a_{k}\delta_{x_k}$ where $\{x_k\}$ is a  sequence in $G$ and $\{a_k\}$ is a   sequence in suitable discrete Lorentz spaces. It was proved by D. Oberlin \cite{Oberlin} for $0<p<1$ and $G $ locally compact, that a convolution operator $T$ is weak $(p,p)$ if and only if $T$ has the form $\sum a_{k}L_{x_k}$ where $L_{x}$ denotes the left-translation at $x$ and the sequence $\{a_k\}$ satisfies $|a_k|=O(k^{-1/p});$ Oberlin also proves that, for $0<p<q\leq 2, $ and $G$  amenable, the weak $(p,p)$ boundedness of a convolution operator implies its $L^q$ boundedness. The same author has characterized in a similar way those bounded operators on $L^p(G),$ $0<p<1$ (see \cite{Oberlin2}) by observing that those convolution operators are given by convolutions with discrete measures $(m_j)\in\ell^p$.

In general, bounded convolution operators on $L^{1,\infty}(G)$ in suitable groups, have a restriction to $L^1(G)$  given by convolution with discrete measures, but, the operators are not uniquely determinate by those measures; this fact was observed by P. Sj\"ogren in \cite{Sjo1} where he also study the form of such measures for $G=\mathbb{R},\mathbb{R}/\mathbb{Z}$ and $G=\mathbb{Z}$. The works Colzani and Sj\"ogren \cite{ColzSjo}  and  Sj\"ogren\cite{Sjo2} give a study of convolution operators on the Lorentz space $L^{1,q},$ $1<q<\infty.$ The classical work Cowling and Fournier \cite{Cowlin} illustrates some inclusions between spaces of convolutors. Other properties of the complex valued and vector valued convolution operators  can be found in \cite{Bakes,Candeal,Candeal2,Lai}.

On several classes of Lie groups, which also are locally compact topological groups, convolution operators appear as multipliers of the Fourier transform. Classical references in this subject are Marcinkiewicz\cite{Marc}, Mihlin \cite{Mihlin}, Alexopoulos\cite{alexo}, Fefferman\cite{Fef}, the seminal work of H\"ormander\cite{Hormander1960}. In a much more recent context, multipliers have been considered in Ruzhansky and Wirth\cite{RW}, Fischer and Ruzhansky\cite{FR}, Fischer\cite{Fischer2},  the works \cite{RR,RR2,RR3} and references therein.

In the next section we prove our main result and we discuss some applications for multiplier on Lie groups.

\section{ $\textnormal{Weak-}\ell^p$ estimates for convolution operators }
In this section we prove our main result. Our main tool will be the following lemma.

\begin{lemma}[Approximate identity]\label{Lemma}
  Let $G$ be a locally compact (Hausdorff) topological group. The exist complex functions $\delta_\varepsilon,$ $\varepsilon>0$ on $G$ satisfying:
  \begin{itemize}
    \item there exists a constant $c>0,$ such that $\Vert \delta_{\varepsilon} \Vert_{L^1(G)}\leq c.$
    \item $\int_{G}\delta_\varepsilon(x)d\mu(x)=1$ for every $\varepsilon>0,$
    \item for every neighborhood $V$ of the identity element $e$ of the group, we have $\int_{G-V}|\delta_\varepsilon (x) |d\mu(x)\rightarrow 0$ as $\varepsilon \rightarrow 0.$
        \item For $1\leq p<\infty,$ $\Vert f\ast \delta_\varepsilon-f \Vert_{L^p(G)}\rightarrow 0$ as $\varepsilon \rightarrow 0.$
            \item If $f$ is continuous in some compact set $K$ of $G,$ then $ f\ast \delta_\varepsilon\rightarrow f$ in $L^{\infty}(K).$
  \end{itemize}
\end{lemma}
Now, we present our generalization of the Sobolev-Hardy-Littlewood-Stepanov Theorem.
\begin{theorem}\label{Thm}
  Let us assume that $G$ is a locally compact (Hausdorff) topological group. Let us consider a measurable function $k$ on $G.$ If $T$ is the convolution operator associated to $k,$ then $T:L^1(G)\rightarrow L^{p,\infty}(G) ,$ $1<p<\infty,$ extends to a bounded operator if and only if $k\in L^{p,\infty}(G)).$
    \end{theorem}
  \begin{proof}
     If we assume $k\in L^{p,\infty}(G)), $ we can write for $f\in L^1(G),$
    \begin{equation}\label{necessary}
      \Vert Tf\Vert_{L^{p,\infty}}=\Vert k\ast f\Vert_{L^{p,\infty}}\leq \Vert k\Vert_{L^{p,\infty}}\Vert f\Vert_{L^1},
    \end{equation}  where we have used the weak-Young inequality (see Theorem 1.2.13 of \cite{GrafakosBook}). For the converse assertion, let us assume that $T$ is a bounded operator of weak type $(1,p).$ Then, we have the inequality
    \begin{equation}\label{sufficienty}
      \Vert Tf\Vert_{L^{p,\infty}}\leq \Vert T\Vert_{\mathscr{B}(L^1, L^{p,\infty})}\Vert f\Vert_{L^1},
    \end{equation} in particular, for every $\varepsilon>0$ and $\delta_\varepsilon$ as in \eqref{Lemma} we have
    \begin{equation}
      \Vert T\delta_\varepsilon \Vert_{L^{p,\infty}}=\Vert k\ast \delta_\varepsilon \Vert_{L^{p,\infty}}\leq \Vert T\Vert_{\mathscr{B}(L^1, L^{p,\infty})}\Vert \delta_\varepsilon\Vert_{L^1}\leq  c\cdot \Vert T\Vert_{\mathscr{B}(L^1, L^{p,\infty})} .
    \end{equation} Consequently we obtain the inequality
    \begin{eqnarray}
    \Vert k\ast \delta_\varepsilon \Vert_{L^{p,\infty}}\leq c\cdot  \Vert T\Vert_{\mathscr{B}(L^1, L^{p,\infty})} ,
    \end{eqnarray}
    uniformly on $\varepsilon.$ The subset $\{k\ast \delta_\varepsilon\}$ is then bounded uniformly in $L^{p,\infty}(G).$ Since $L^{p,\infty}(G)$ is the dual of the Lorentz space $L^{p',1}(G)$ (see Theorem 1.4.17 of \cite{GrafakosBook}, where $p'$ is the conjugate exponent of $p$) the functions $k\ast \delta_\varepsilon$  are contained in a fixed multiple of the unit ball in $ L^{p',1}(G)^*\equiv  L^{p,\infty}(G).$ By the Banach-Alaoglu theorem this is a compact set in the $\textnormal{weak}^*$ topology. Then, there exists a subsequence $\varepsilon_k\rightarrow 0,$ such that the sequence $k\ast \delta_{\varepsilon_k}$ converges to some function $\varkappa \in  L^{p',1}(G)^*\equiv  L^{p,\infty}(G), $ in the $\textnormal{weak}^*$ topology. So, we have
    \begin{equation}\label{eq}
      \lim_{k\rightarrow\infty}\int\limits_{G}g(x)k\ast \delta_{\varepsilon_k}(x)d\mu(x)=\int\limits_{G}g(x)\varkappa(x) d\mu(x),\,\,g\in L^{p',1}(G).
    \end{equation}In order to end the proof we only need to show that $k=\varkappa.$ In fact,  by properties of the convolution, we have for a real function   $h\in C_0(G),$
    \begin{equation}\label{convo}
      (k*\delta_{\varepsilon_k},h)_{L^2}=(k,\delta_{\varepsilon_k}*h )_{L^2}.
    \end{equation}
    Because $\delta_{\varepsilon_k}*h$ converges to $h$ in $C_0(G),$ we have $(k,\delta_{\varepsilon_k}*h )_{L^2}\rightarrow (k,h)_{L^2}.$ On the other hand, if $k\rightarrow \infty,$ $(k*\delta_{\varepsilon_k},h)_{L^2}\rightarrow (\varkappa,h)_{L^2}.$ So, we have
    \begin{equation}\label{ttttttt}
      (k,h)_{L^2}=(\varkappa,h)_{L^2},\,\,h\in C_0(G),
    \end{equation}
    and consequently $k=\varkappa,$ a.e.w. Thus, we end the proof of the theorem.
\end{proof}
\begin{remark}
Although our main result has been announced for (left) convolution operator of the form
\begin{equation}\label{T}
  Tf(x)\equiv  k*f(x):=\int\limits_{G}k(y)f(y^{-1}x)d\mu(y),\,\,f\in C_{0}(G),
\end{equation} the assertion remains valid for (right) convolution operators
\begin{equation}\label{T2}
  Tf(x)\equiv  f*k(x):=\int\limits_{G}f(y)k(y^{-1}x)d\mu(y),\,\,f\in C_{0}(G).
\end{equation}
\end{remark}
We end this paper with natural examples where our main result can be useful.
\begin{example} Let us consider the discrete group $G=\mathbb{Z}^n,$ in this case, the Sobolev-Hardy-Littlewood-Stepanov theorem has been proved in \cite{CardonaZn}. Convolution operators in this setting take the form
\begin{equation}\label{zn}
  tf(s)=\sum_{n'\in\mathbb{Z}^n}k(s-n')f(n'). \end{equation}
 By using the Fourier transform $\mathscr{F}_{\mathbb{Z}^n}$ on $\mathbb{Z}^n,$ convolution operators are multipliers:
  \begin{equation}\label{multiplierzn}
    tf(s)=\int_{[0,1]^n}e^{i2\pi s\theta}m(\theta)(\mathscr{F}_{\mathbb{Z}^n}f)(\theta)d\theta,
  \end{equation} with $m=\mathscr{F}_{\mathbb{Z}^n}k.$ Now, $t$ is of weak type $(1,p),$ $1<p<\infty$ if and only if $k\in \ell^{p,\infty}(\mathbb{Z}^n).$
\end{example}
\begin{example}
Let us assume that $G$ is a compact Lie group (examples of compact Lie groups are $n$-dimensional torus $G=\mathbb{T}^n$ or $G=\textnormal{SU}(2)$), and $\widehat{G}$ is its unitary dual. If $\mathscr{F}_{G}$ is the Fourier transform of the group, convolution operators appear as multipliers of the form
\begin{equation}\label{multigcompact}
  T_{\sigma}f(x)=\sum_{\pi\in \widehat{G}}d_\pi \textnormal{tr}[\pi(x)\sigma(\pi)(\mathscr{F}_{G}f(\pi))].
\end{equation}
\end{example}For non-compact Lie groups as nilpotent Lie groups (natural examples are $G=\mathbb{R}^n$ or $G=\mathbb{H}^n$ which denotes the Heisenberg group), multipliers appear in the form
\begin{equation}\label{multigncompact}
  T_{\sigma}f(x)=\int\limits_{ \widehat{G}}\textnormal{tr}[\pi(x)\sigma(\pi)(\mathscr{F}_{G}f(\pi))]d\pi,
\end{equation} where $d\pi$ is the Plancherel measure on $\widehat{G}.$ The convolution kernel of $T$ in both cases, is given by $k=\mathscr{F}_{G}^{-1}\sigma,$ and we also have in both cases that $T_\sigma$ is of weak type $(1,p)$ if and only if $k\in L^{p,\infty}.$ In terms of the so called symbol $\sigma$ of the operator $T_\sigma,$ the Fourier multipier $T_\sigma$ is of weak type $(1,p)$ if and only if $\sigma\in \mathscr{F}_{G}(L^{p,\infty}(G)). $ For more properties about multipliers in compact Lie groups, and graded Lie groups we refer the reader to the references Fischer and Ruzhansky\cite{FR2} and Ruzhansky and Turunen \cite{Ruz}.

\bibliographystyle{amsplain}

\end{document}